\documentclass[preprint,11pt]{elsarticle}
\usepackage{lineno,hyperref}
\modulolinenumbers[5]

\usepackage{color}
\usepackage{amsmath}
\usepackage{amsfonts,amsthm,amssymb}
\usepackage{amsfonts}
\usepackage{graphics}
\usepackage{pstricks}
\usepackage{graphicx}
\usepackage{cleveref}
\usepackage{extarrows,chngpage,array,float}
\usepackage{amssymb}
\usepackage{latexsym,bm}
\newtheorem{theorem}{Theorem}[section]
\newtheorem{lemma}[theorem]{Lemma}
\newtheorem{definition}[theorem]{Definition}
\newtheorem{proposition}[theorem]{Proposition}
\newtheorem{corollary}[theorem]{Corollary}

\newtheorem{problem}[theorem]{Problem}

\allowdisplaybreaks   

\journal{xxx}

\newcommand \equ[2]
{
	\begin{equation}\label{#1}
		#2
	\end{equation}
}

\newcommand \eqn[2]
{
	\begin{eqnarray}\label{#1}
		#2
	\end{eqnarray}
}

\begin{document}

\begin{frontmatter}

\title{On the maximum local mean order of sub-$k$-trees of a $k$-tree}

\author{Zhuo Li$^{1}$, \ \ Tianlong Ma$^{1}$, \ \ Fengming Dong$^{2}$, \ \ Xian'an Jin$^{1}$\footnote{Corresponding author.}\\
\small $^{1}$School of Mathematical Sciences\\[-0.8ex]
\small Xiamen University\\[-0.8ex]
\small P. R. China\\
\small $^{2}$National Institute of Education\\[-0.8ex]
\small Nanyang Technological University\\[-0.8ex]
\small Singapore\\
\small\tt Email: zhuoli2000@aliyun.com, tianlongma@aliyun.com, fengming.dong@nie.edu.sg, xajin@xmu.edu.cn}

\begin{abstract}
For a $k$-tree $T$, a generalization of a tree, the local mean order of sub-$k$-trees of $T$ is the average order of sub-$k$-trees of $T$ containing a given $k$-clique. The problem whether the largest local mean order of a tree (i.e., a $1$-tree) at a vertex always takes on at a leaf was asked by Jamison in 1984 and was answered by Wagner and Wang in 2016. In 2018, Stephens and Oellermann asked a similar problem: for any $k$-tree $T$, does the maximum local mean order of sub-$k$-trees containing a given $k$-clique occur at a $k$-clique that is not a major $k$-clique of $T$? In this paper, we give it an affirmative answer.
\end{abstract}


\end{frontmatter}

\section{Introduction}
As the analog of trees in higher dimensions, Beineke and Pippert introduced the concept of $k$-trees in \cite{Beineke} for any integer $k\ge 1$, which can be defined recursively as follows.
\begin{definition}
Fix an integer $k\geq1$.
\begin{description}
\item[(1)] The complete graph $K_k$ is a $k$-tree.
\item[(2)] If $T$ is a $k$-tree, then so is the graph obtained from $T$ by adding a new vertex adjacent to all vertices of some $k$-clique of $T$.
\end{description}
\end{definition}

We use $T$ to denote a $k$-tree of order $n$ and $C$ to denote a $k$-clique in $T$. A $k$-tree $T$ is \emph{trivial} if $T\cong K_k$ and \emph{non-trivial} otherwise. A \emph{sub-$k$-tree} of a $k$-tree $T$ is a subgraph of $T$ that is itself a $k$-tree. A $1$-tree yields precisely a tree and a sub-$1$-tree yields a subtree.

The mean order of subtrees and the subtree polynomial of a tree were first studied by Jamison \cite{Jamison1,Jamison2} in the 1980's. In 2018, Stephens and Oellermann \cite{Stephens} extended these concepts from trees to $k$-trees.

Let $a_i(T)$ denote the number of sub-$k$-trees of order $i$ in $T$, where $i\ge k$.
The \emph{sub-$k$-tree polynomial} of a $k$-tree $T$, denoted by $\Phi_T (x)$, is the generating function
for $a_i(T)$, that is,
\[\Phi_T (x)=\sum_{i=k}^{n}a_i(T)x^i.\]
The \textit{global mean order} of sub-$k$-trees of a $k$-tree $T$ is defined by
\begin{align}\notag
\mu(T)= \frac{\Phi'_T (1)}{\Phi_T (1)},
\end{align}
that is, the average order of all sub-$k$-trees of $T$.

For a fixed sub-$k$-tree $S$ of a $k$-tree $T$, let $a_i(T;S)$ denote the number of sub-$k$-trees of order $i$ in $T$ containing $S$. The \emph{local sub-$k$-tree polynomial} of $T$ at $S$ is defined by
\[\phi_{T,S}(x)=\sum_{i=|V(S)|}^{n}a_i(T;S)x^i.\]
Correspondingly, the \emph{local mean order} of sub-$k$-trees containing a given sub-$k$-tree $S$ of a $k$-tree $T$ is defined by
\begin{align}\notag
\mu(T;S)= \frac{\phi'_{T,S} (1)}{\phi_{T,S} (1)},
\end{align}
that is, the average order of all sub-$k$-trees of $T$ containing $S$.

There are many studies on the mean order of subtrees of a tree. In 1983, Jamison \cite{Jamison1} proved that the global mean order of a tree of order $n$ is at least $\frac{n+2}{3}$, and equality is achieved only by a path. In 1984, Jamison \cite{Jamison2} studied the monotonicity problem and proved that the global mean order of a tree becomes smaller after contracting its any edge. Recently, Luo, Xu, Wagner and Wang \cite{Luo} further proved that the global mean order of a tree $T$ does not decrease by more than $\frac{1}{3}$ after contracting a pendant edge. Several open problems and conjectures posed in \cite{Jamison1,Jamison2} were subsequently solved or partially solved in \cite{Cambie, Haslegrave, Luo, Mol, Vince, Wagner1, Wagner2}. In particular, Jamison \cite{Jamison1} asked whether the largest local mean order of subtrees of a tree at a vertex always takes on at a leaf for any tree. In 2016, Wagner and Wang \cite{Wagner2} studied this problem and obtained the following result.

\begin{theorem}[\cite{Wagner2}]\label{Theorem 1.2}
The maximum local mean order of subtrees of a tree at a vertex occurs either at a leaf or at a vertex of degree 2.
\end{theorem}

For a $k$-clique $C$ in a $k$-tree $T$, the \emph{degree} of $C$, denoted by $deg_T(C)$, is defined to be the number of $(k+1)$-cliques of $T$ containing $C$. A \emph{major} $k$-clique is a $k$-clique in $T$ with degree at least $3$. In this paper we are mainly interested in the following problem posed by Stephen and Oellermann in \cite{Stephens}, which can be viewed as a generalization of Theorem \ref{Theorem 1.2}.

\begin{problem}[\cite{Stephens}, \textbf{Problem 4}]\label{mainproblem}
For any $k$-tree $T$, does the maximum local mean order of sub-$k$-tree containing a given $k$-clique occur at a $k$-clique that is not a major $k$-clique of $T$?
\end{problem}

Motivated by this problem, we first study the effect of the Kelmans operation on the local mean order of a tree at a vertex.  Then by converting a $k$-tree to its characteristic $1$-tree and using the partial Kelamans operation on trees, we obtain an affirmative answer to Problem \ref{mainproblem}.

\begin{theorem}\label{MainTheorem}
For any $k$-tree $T$, the maximum local mean order of sub-$k$-tree containing a given $k$-clique occurs at a $k$-clique that is not a major $k$-clique of $T$.
\end{theorem}

The rest of this paper is organized as follows. In Section 2, we make necessary preparations for Section 3, mainly obtaining a formula for the local mean order of a tree at a vertex. In Section 3, we compare the local mean orders of subtrees of two trees related by a Kelmans operation or a partial Kelmans operation. In section 4, we prove the main Theorem \ref{MainTheorem}. In Section 5, as by-products two known results obtained in \cite{Wagner2} and \cite{Jamison1} are derived and Theorem \ref{MainTheorem} are further discussed.

\section{Preliminaries}

The following lemma in \cite{Jamison1}
follows immediately from the definition of the subtree polynomial.

\begin{lemma}[\cite{Jamison1}]\label{lemma2.0}
Let $T$ be a tree, $u$ be a vertex of $T$, and $T^*_w$ be the connected component of $T-u$ containing $w$. Then the subtree polynomial of $T$ at $u$
\equ{eq2-0}
{
\phi_{T,u}(x)=x\prod_{w\in N_T(u)}(1+\phi_{T^*_w,w}(x)),
}
where $N_T(u)$ is the set of neighbors of $u$ in $T$.
\end{lemma}

Let $u$ and $v$ be two adjacent vertices in a tree $T$,
and let $N_T(u,v)=(N_T(u)\cup N_T(v))\setminus\{u,v\}$.
Since $N_T(u)\cap N_T(v)=\emptyset$, $N_T(u,v)$ is the disjoint union
of $N_T(u)\setminus \{v\}$
and $N_T(v)\setminus \{u\}$.
For $w\in N_T(u,v)$, let $T_{w}$ denote the connected component containing $w$ in $T-\{u,v\}$.
Obviously,  for any
$w\in N_T(u)\setminus \{v\}$,
$T_w$ and $T^*_w$ are the same subtree
of $T$.
Applying Lemma \ref{lemma2.0},
we obtain the following expression
for  $\phi_{T,u}(x)$:
\eqn{eq2-1}
{
	\phi_{T,u}(x)&=&x(1+\phi_{T^*_v,v}(x))\prod_{w\in N_T(u)\setminus \{v\}}(1+\phi_{T^*_w,w}(x))
	\nonumber 	\\
	&=&x(1+x\prod_{w\in N_T(v)\setminus \{u\}}(1+\phi_{T_w,w}(x)))\prod_{w\in N_T(u)\setminus \{v\}}(1+\phi_{T_w,w}(x))\notag\\
	&=&x\prod_{w\in N_T(u)\setminus \{v\}}(1+\phi_{T_w,w}(x))+x^2\prod_{w\in N_T(u,v)}(1+\phi_{T_w,w}(x)).
}

Thus, its derivative
\eqn{eq2-2}
{
\phi'_{T,u}(x)
&=&(1+x\sum_{w\in N_T(u)\setminus \{v\}}{{\phi'_{T_w,w}(x)}\over{1+\phi_{T_w,w}(x)}})\prod_{w\in N_T(u)\setminus \{v\}}(1+\phi_{T_w,w}(x))\notag\\
& &+(2x+x^2\sum_{w\in N_T(u,v)}{{\phi'_{T_w,w}(x)}\over{1+\phi_{T_w,w}(x)}})\prod_{w\in N_T(u,v)}(1+\phi_{T_w,w}(x)).
}
Hence,
\equ{eq2-3}
{
\phi_{T,u}(1)\!=\!\prod_{w\in N_T(u)\setminus \{v\}}\!(1+\phi_{{T_w},w}(1))\!
+\!\prod_{w\in N_T(u,v)}\!
(1+\phi_{{T_w},w}(1)),
}
and
\eqn{eq2-4}
{
\phi'_{T,u}(1)&=&\left (1+\sum_{w\in N_T(u)\setminus \{v\}}{{\phi'_{T_w,w}(1)}
	\over{1+\phi_{T_w,w}(1)}}\right )
\prod_{w\in N_T(u)\setminus \{v\}}(1+\phi_{T_w,w}(1))\notag\\
& &+\left (2+\sum_{w\in N_T(u,v)}{{\phi'_{T_w,w}(1)}\over{1+\phi_{T_w,w}(1)}}\right )
\prod_{w\in N_T(u,v)}(1+\phi_{T_w,w}(1)).
}
By (\ref{eq2-3}), (\ref{eq2-4}) and the definition of $\mu(T; u)$,
we have
\begin{align}\notag\label{eq1}
\mu(T;u)=&\frac{\phi'_{T,u}(1)}{\phi_{T,u}(1)}\\\notag
=&\dfrac{1\!+\!\sum\limits_{w\in N_T(u)\setminus \{v\}}\frac{\phi'_{T_w,w}(1)}{1+\phi_{T_w,w}(1)}}{1+\prod\limits_{w\in N_T(v)\setminus \{u\}}\!(1+\phi_{{T_w},w}(1))}\\
&\!+\!\dfrac{(2+\!\sum\limits_{w\in N_T(u,v)}\!\frac{\phi'_{T_w,w}(1)}{1+\phi_{T_w,w}(1)})\prod\limits_{w\in N_T(v)\setminus \{u\}}\!(1+\phi_{{T_w},w}(1))\!}{1+\prod\limits_{w\in N_T(v)\setminus \{u\}}(1+\phi_{{T_w},w}(1))}.
\end{align}

In the next section, we will apply a relation between $\alpha_i=\phi_{T_{u_i},u_i}(1)$ and $\alpha'_i=\phi'_{T_{u_i},u_i}(1)$ due to Jamison \cite{Jamison1}, as stated below. See also \cite{Wagner2}.

\begin{lemma}[\cite{Jamison1}]
	\label{lemma2.1}
	For any tree $T$ and any vertex $u$ of $T$, we have
	\begin{align}
		\notag {\phi'_{T,u}(1)\over 1+\phi_{T,u}(1)}\leq\frac{\phi_{T,u}(1)}{2},
	\end{align}
	with equality if and only if $T$ is a path and $u$ is a leaf of $T$ (which includes the case that $T$ is trivial, i.e., $T=u$).
\end{lemma}

\section{Local mean order under Kelmans operation}

In this section, we study the change of local mean order of subtrees of a tree at a vertex under a Kelmans operation or a partial Kelmans operation.

For a vertex $v$ of a graph $G$, recall that $N_G(v)$ denote the set of neighbours of $v$ in $G$,
Let $N_G[v]=N_G(v)\cup \{v\}$.
For two vertices $u$ and $v$
of $G$,  let $N_1=N_G(u)\setminus N_G[v]$ and  $N_2=N_G(v)\setminus N_G[u]$.

\begin{definition}
	The Kelmans operation on $G$ from $v$ to $u$ is defined to be the operation which replaces the edge $vw$ by a new edge $uw$ for each vertex $w\in N_2$.
\end{definition}

\newcommand \Kel[2]
{
	{}_{(#1\rightarrow #2)}
}

Let $G\Kel{v}{u}$ denote the graph obtained by applying the Kelmans operation on $G$ from $v$ to $u$. It is not difficult to see that
$G\Kel{v}{u}\cong G$ if and only if
 either $N_1=\emptyset$ or $N_2=\emptyset$.
Furthermore, the two graphs
	$G\Kel{v}{u}$ and $G\Kel{u}{v}$
	are always isomorphic.

Let $W\subseteq N_2$. If we only replace the edge $vw$ by a new edge $uw$ for each $w\in W$, this operation will be called a \emph{partial Kelmans operation}. A partial Kelmans operation on a tree $T$ from $v$ to one of its adjacent vertices $u$ is shown in Figure \ref{fig1}.

\begin{figure}[!ht]
	\centering
	\includegraphics[width=0.8\linewidth]{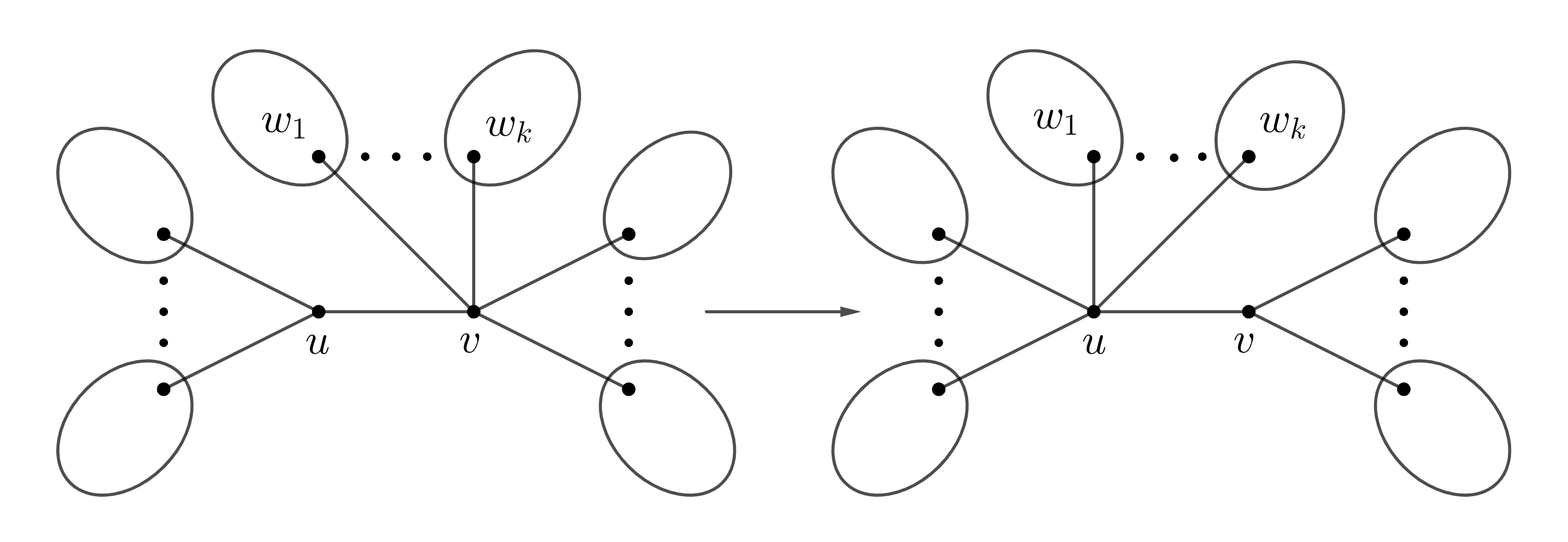}
	\caption{A partial Kelmans operation on a tree from $v$ to $u$.}
	\label{fig1}
\end{figure}

\begin{theorem}\label{mainKO}
Let $u$ and $v$ be two adjacent vertices in a tree $T$. Then
\begin{align}\label{eq2}
\mu(T\Kel{v}{u};v)\geq\mu(T;u),
\end{align}
and
\begin{align}\label{eq3}
\mu(T;v)\geq  \mu(T\Kel{v}{u};u).
\end{align}
Moreover, equality holds in (\ref{eq2}) if and only if either $u$ is a leaf or
$T$ is a path with $v$ as a leaf, and equality holds in (\ref{eq3}) if and only if the component containing $u$ in $T-v$ is a path with $u$ as its leaf (which includes the case that this component has only one vertex $u$).
\end{theorem}
\begin{proof}
Suppose that
\[N_T(u)\setminus \{v\}=\{u_1,u_2,\cdots,u_k\}
\quad \mbox{and}\quad
N_T(v)\setminus \{u\}=\{v_1,v_2,\cdots,v_s\}.\]
Note that either $N_T(u)\setminus \{v\}$ or $N_T(v)\setminus \{u\}$ may be empty, namely, $k=0$ or $s=0$. Then $N_T(u,v)=\{u_1,u_2,\cdots,u_k,v_1,v_2,\cdots,v_s\}$. Recall that for each $w\in N_T(u,v)$, $T_w$ is the connected component of $T-\{u,v\}$ containing $w$. For simplicity, we set $\alpha_i=\phi_{T_{u_i},u_i}(1)$, $\beta_i=\phi_{T_{v_i},v_i}(1)$, $\alpha_i'=\phi'_{T_{u_i},u_i}(1)$,
$\beta_i'=\phi'_{T_{v_i},v_i}(1)$, and let
\equ{eq3-1}
{
\alpha=\prod_{i=1}^{k}(1+\alpha_i)
\quad \mbox{and}\quad
 \beta=\prod_{i=1}^{s}(1+\beta_i).
}
Obviously,
$\alpha\ge 1$ (resp. $\beta\ge 1$)
with equality if and only if $k=0$
(resp. $s=0$).
Furthermore,
$\alpha \ge 1
+\sum\limits_{i=1}^k \alpha_i$,
where the equality holds if and only if
$k\le 1$.
Analogous property also holds for $\beta$.

By (\ref{eq1}) and (\ref{eq3-1}),
we have
\begin{align}\label{eq5}
\mu(T;u)=\dfrac{1+\sum_{i=1}^{k}\frac{\alpha'_i}{1+\alpha_i}+2\beta+\beta\left(\sum_{i=1}^{k}\frac{\alpha'_i}{1+\alpha_i}+\sum_{i=1}^{s}\frac{\beta'_i}{1+\beta_i}\right)}{1+\beta}.
\end{align}
and
\begin{align}\label{eq6}
\mu(T;v)=\dfrac{1+\sum_{i=1}^{s}\frac{\beta'_i}{1+\beta_i}+2\alpha+\alpha\left(\sum_{i=1}^{k}\frac{\alpha'_i}{1+\alpha_i}+\sum_{i=1}^{s}\frac{\beta'_i}{1+\beta_i}\right)}{1+\alpha}.
\end{align}

Note that $T\Kel{v}{u}$ is obtained by the Kelmans operation on $T$ from $v$ to $u$. We have
\equ{eq3-2}
{
N_{T\Kel{v}{u}}(u)\setminus \{v\}
=\{u_i:1\le i\le k\}\cup \{v_j: 1\le j\le s\}\
\mbox{and}
\
N_{T\Kel{v}{u}}(v)\setminus \{u\}=\emptyset.
}

Similarly, by (\ref{eq1}) and (\ref{eq3-1}),
we have
\begin{align}\label{eq7}
\mu(T\Kel{v}{u};v)=\frac{1+2\alpha\beta+\alpha\beta\left(\sum_{i=1}^{k}\frac{\alpha'_i}{1+\alpha_i}+\sum_{i=1}^{s}\frac{\beta'_i}{1+\beta_i}\right)}{1+\alpha\beta},
\end{align}
and
\begin{align}\label{eq8}
\mu(T\Kel{v}{u};u)=\frac{3}{2}+\sum_{i=1}^{k}\frac{\alpha'_i}{1+\alpha_i}+\sum_{i=1}^{s}\frac{\beta'_i}{1+\beta_i}.
\end{align}
We first prove (\ref{eq2}). By (\ref{eq5}) and (\ref{eq7}), after a simple calculation, we have
\eqn{eq3-3}
{
& &\mu(T\Kel{v}{u};v)-\mu(T;u)\notag\\
&=&\frac{\alpha\beta+\alpha\beta\sum_{i=1}^{s}\frac{\beta'_i}{1+\beta_i}-\sum_{i=1}^{k}\frac{\alpha'_i}{1+\alpha_i}-\beta-\beta\left(\sum_{i=1}^{k}\frac{\alpha'_i}{1+\alpha_i}+\sum_{i=1}^{s}\frac{\beta'_i}{1+\beta_i}\right)}{(1+\alpha\beta)(1+\beta)}
\notag \\
&=&\frac{\beta \gamma}
{(1+\alpha\beta)(1+\beta)},
}
where
\eqn{eq3-4}
{
\gamma&=&\alpha+\alpha\sum_{i=1}^{s}\frac{\beta'_i}{1+\beta_i}-\frac{1}{\beta}\sum_{i=1}^{k}\frac{\alpha'_i}{1+\alpha_i}-1-\sum_{i=1}^{k}\frac{\alpha'_i}{1+\alpha_i}-\sum_{i=1}^{s}\frac{\beta'_i}{1+\beta_i}
\notag\\
&=&\alpha-1+(\alpha-1)\sum_{i=1}^{s}\frac{\beta'_i}{1+\beta_i}-(1+{1\over \beta})\sum_{i=1}^{k}{\alpha'_i\over 1+\alpha_i}.
}
Observe that
	$\mu(T\Kel{v}{u};v)=\mu(T;u)$ if and only if $\gamma=0$.
	If $k=0$ (i.e., $u$ is a leaf), then $\alpha=1$ and $\gamma=0$ by (\ref{eq3-4}).
Now assume that $k\ge 1$.
By Lemma \ref{lemma2.1},
for each $i=1,2,\cdots, k$, we have
\begin{align}\label{eq4}
	\frac{\alpha'_i}{1+\alpha_i}\leq \frac{\alpha_i}{2},
\end{align}
with equality if and only if $T_{u_i}$ is a path with $u_i$ its leaf.
By (\ref{eq4}), we have

\eqn{eq3-5}
{
\gamma
&\geq & \alpha-1+(\alpha-1)\sum_{i=1}^{s}\frac{\beta'_i}{1+\beta_i}-\left(\frac{1}{2}+\frac{1}{2\beta}\right)\sum_{i=1}^{k}\alpha_i
\\
&\ge &
\alpha-1-\sum_{i=1}^{k}\alpha_i
\label{eq3-6}
\\
&\ge &0,  \label{eq3-7}
}

where the equality of (\ref{eq3-5})
holds if and only if $T_{u_i}$ is a path with $u_i$ as its leaf for each $i=1,2,\cdots,k$,
the equality of (\ref{eq3-6})
holds if and only if $s=0$,
and
the equality of (\ref{eq3-7})
holds if and only if $k=1$.

Observe that
$s=0$, $k=1$
and $T_{u_1}$ is a path with $u_1$
as a leaf
if and only if  $T$ is  a path with $v$ as a leaf.

Now we are going to prove  (\ref{eq3}).
By (\ref{eq6}) and (\ref{eq8}), we have

\eqn{eq3-8}
{
\mu(T;v)-\mu(T\Kel{v}{u};u)
&=&\frac{\alpha+\alpha\sum_{i=1}^{k}\frac{\alpha'_i}{1+\alpha_i}}{1+\alpha}-\frac{1}{2}-\sum_{i=1}^{k}\frac{\alpha'_i}{1+\alpha_i}\notag\\
&=&\frac{1}{1+\alpha}\left(
\frac{1}{2}\alpha-\frac 12 -
\sum_{i=1}^{k}\frac{\alpha'_i}{1+\alpha_i}
\right)
\notag\\
&\ge &
\frac{1}{1+\alpha}
\left(\frac{1}{2}\alpha
-\frac 12
-\sum_{i=1}^{k}\frac{\alpha_i}{2}\right)
\\
&\ge &0,
\label{eq3-9}
}
where the inequality of (\ref{eq3-8})
follows from (\ref{eq4}),
and the inequality of (\ref{eq3-9})
follows from the fact that
$\alpha\ge 1
+\sum\limits_{i=1}^k \alpha_i$.

Thus (\ref{eq3}) holds with equality if and only if $k=0$ or $k=1$ and $T_{u_1}$ is a path with $u_1$ its leaf, which means the component containing $u$ in $T-v$ is a path with $u$ as its leaf.
\end{proof}

\begin{corollary}\label{jin}
Let $u$ and $v$ be two adjacent vertices of a tree $T$. Then

\equ{eq3-10}
{
\mu(T\Kel{v}{u};v)\geq\mu(T;v),	
}
with equality if and only if
either $v$ is a leaf of $T$
or
$T$ is a path with $u$ as a leaf.
\end{corollary}

\begin{proof}
By (\ref{eq2}), we have $\mu(T\Kel{u}{v};u)\geq \mu(T;v)$ with equality if and only if
either $v$ is a leaf or
$T$ is a path with $u$ as a leaf.
Since $T\Kel{v}{u}\cong T\Kel{u}{v}$
with an isomorphism
from $T\Kel{v}{u}$ to $T\Kel{u}{v}$
mapping $v$ to $u$ and $u$ to $v$,
we have $\mu(T\Kel{v}{u};v)=
\mu(T\Kel{u}{v};u)$.
Thus, $\mu(T\Kel{v}{u};v)\geq \mu(T;v)$ with equality if and only if
either $v$ is a leaf or $T$ is a path with $u$ as a leaf.
\end{proof}

By generalizing Corollary \ref{jin} from Kelmans operation to partial Kelmans operation, we obtain the following theorem, which will be used to prove Theorem \ref{MainTheorem} in the next section.

\begin{theorem}\label{patialtheorem}
Let $u$ and $v$ be two adjacent vertices of a tree $T$, and let $T'$ be the tree obtained by a partial Kelmans operation on $T$ from $v$ to $u$.
Then
\begin{align}\notag
\mu(T';v)\geq\mu(T;v).
\end{align}
Moreover, the equality holds in the  inequality above if and only if either

\begin{enumerate}
\item[(a)] $T'=T$ (no operation is made), or
\item[(b)] $u$ is a leaf of $T$,
$T'$ is the tree obtained by replacing the edge $vv_1$ by $uv_1$,
where $v_1\in N_T(v)\setminus\{u\}$
and $T_{v_1}$ is a path with $v_1$ as a leaf.
\end{enumerate}
\end{theorem}

\begin{proof}
Let
\[N_T(u)\setminus \{v\}=\{u_1,u_2,\cdots,u_k\},\]
and
\[N_T(v)\setminus \{u\}=\{v_1,v_2,\cdots,v_s,w_1,w_2,\cdots,w_t\}.\]
Let $T'$ denote the tree
obtained by
the partial Kelmans operation on $T$ from $v$ to $u$, which replaces the edge $vv_i$ with the edge $uv_i$ for $i=1,2,\cdots,s$.
If $s=0$, then no operation is really made, $T'=T$, and hence $\mu(T';v)=\mu(T,v)$ holds.
Thus, we need only to consider the case
$s\ge 1$.

For simplicity, we set $\alpha_i=\phi_{T_{u_i},u_i}(1)$, $\beta_i=\phi_{T_{v_i},v_i}(1)$, $\gamma_i=\phi_{T_{w_i},w_i}(1)$, $\alpha_i'=\phi'_{T_{u_i},u_i}(1)$, $\beta_i'=\phi'_{T_{v_i},v_i}(1)$, $\gamma_i'=\phi'_{T_{w_i},w_i}(1)$ and
$$
\alpha=\prod_{i=1}^{k}(1+\alpha_i),\quad
\beta=\prod_{i=1}^{s}(1+\beta_i),\quad
\gamma=\prod_{i=1}^{t}(1+\gamma_i)
$$
and
$$
\delta=\sum_{i=1}^{k}\frac{\alpha_i'}{1+\alpha_i},\quad
\theta=\sum_{i=1}^{s}\frac{\beta_i'}{1+\beta_i},\quad
\omega=\sum_{i=1}^{t}\frac{\gamma_i'}{1+\gamma_i}.
$$
Then, by (\ref{eq1}), we have
\equ{eq3-11}
{
\mu(T;v)=\dfrac{1+\theta+\omega+2\alpha+\alpha(\delta+\theta+\omega)}{1+\alpha}.
}

Note that
\[N_{T'}(u)\setminus \{v\}=\{u_1,u_2,\dots,u_k,v_1,v_2,\dots,v_s\}\]
and
\[N_{T'}(v)\setminus \{u\}=\{w_1,w_2,\dots,w_t\}.
\]
By (\ref{eq1}), we also have
\equ{eq3-12}
{
\mu(T';v)=\dfrac{1+\omega+2\alpha\beta+\alpha\beta(\delta+\theta+\omega)}{1+\alpha\beta}.
}
By a simple calculation, we have
\eqn{eq3-13}
{
\mu(T';v)-\mu(T;v)
&=&\alpha\dfrac{\beta+(\beta-1)\delta-(\frac{1}{\alpha}+1)\theta-1}
{(1+\alpha)(1+\alpha\beta)}
\notag \\
&\geq &
 \frac{\alpha(\beta -2\theta-1)}
 {(1+\alpha)(1+\alpha\beta)},
}
where the equality holds
if and only if
$\alpha=1$ and $(\beta-1)\delta=0$,
i.e., $k=0$.

By Lemma \ref{lemma2.1},
we have
\equ{eq3-14}
{
\beta - 2\theta-1
=\beta - \sum_{i=1}^{s}\frac{2\beta_i'}{1+\beta_i}
-1
\geq \beta-\sum_{i=1}^{s}\beta_i-1,
}
where the  equality holds if and only if
$T_{v_i}$ is a path with leaf $v_i$ for each $1\leq i\leq s$.

It is easy to see $\beta-\sum_{i=1}^{s}\beta_i-1\geq0$,
where the equality holds
if and only if $s=1$.
By (\ref{eq3-13}) and (\ref{eq3-14}),
we have $\mu(T';v)-\mu(T;v)\geq0$,
where the equality holds
if and only if
$k=0$, $s=1$ and $T_{v_1}$ is a path with $v_1$ as a leaf.
The conclusion holds.
\end{proof}

\section{Proof of main Theorem}

In this section, we prove Theorem \ref{MainTheorem} by using the partial Kelmans operation.
We first explain the characteristic $1$-tree to convert a $k$-tree to a $1$-tree. Note that any vertex in a $k$-tree of order $n\geq k+1$ has degree at least $k$. A vertex with degree $k$ in a $k$-tree $T$
of order $n\geq k+1$ is called a \emph{$k$-leaf} of $T$.
Thus a $1$-leaf is a leaf of a tree. The set of $k$-leaves of a $k$-tree $T$ will be denoted by $L(T)$.
A $k$-clique containing a
$k$-leaf in a $k$-tree is called a $simplicial$ $k$-clique.

\begin{definition} (path-type $k$-tree) Fix an integer $k\geq 1$.
\begin{description}
\item[(1)] The complete graphs $K_k$ and $K_{k+1}$ are path-type $k$-trees.
\item[(2)] If $P$ is a path-type $k$-tree of order at least $k+1$, then so is the graph obtained from $P$ by joining a new vertex to some simplicial $k$-clique of $P$.
\end{description}
\end{definition}
Note that a $k$-tree not $K_k$ or $K_{k+1}$ is a path-type $k$-tree if and only if it has exactly two $k$-leaves.

For each non-trivial $k$-tree $T$ with a given $k$-clique $C$, if the sequence $(v_1,v_2\cdots,v_p)$ of vertices in $T$ satisfies the following conditions:
\begin{description}
\item[(1)] $\{v_1,v_2\cdots,v_p\}\cup V(C)=V(T)$; and
\item[(2)] $v_i$ is a $k$-leaf in $T-\{v_j:1\le j\le i-1\}$ for
all $i$ with $1\le i\le p$,
\end{description}
then this sequence is called a \emph{perfect elimination ordering} of $T$ down to $C$.

The following result,
due to Stephens and Oellermann~\cite{Stephens},
states that there is a unique path-type $k$-tree starting at any given $k$-clique $C$ and ending at any vertex $v$ in $V(T)\setminus V(C)$.

\begin{lemma}[\cite{Stephens}]
	\label{uniq}
        Let $T$ be a $k$-tree with a $k$-clique $C$ and $v\in V(T)\setminus V(C)$. Then there exists a unique sequence $A_T(C,v)=(C,w_1,w_2,\cdots,w_s,v)$ with $w_i\in V(T)\setminus V(C)$ ($1\le i\le s$) such that
        \begin{description}
                \item[(1)] the graph induced by $V(C)\cup\{w_1,w_2,\cdots,w_s,v\}$, denoted by $P_T(C,v)$, is a path-type $k$-tree with $C$
                as a simplicial $k$-clique
                and $v$ as a $k$-leaf; and
                \item[(2)]  the sequence $(v,w_s,w_{s-1},\cdots,w_1)$ is a perfect elimination ordering of $P_T(C,v)$ down to $C$.
        \end{description}
\end{lemma}
Figure \ref{fig3} shows an example of a $2$-tree and a path-type $2$-tree $P_T(C,v_4)$ from \cite{Stephens}, where $C=\{c_1,c_2\}$.
\begin{figure}[!ht]
	\centering
	\includegraphics
	[width=10 cm]{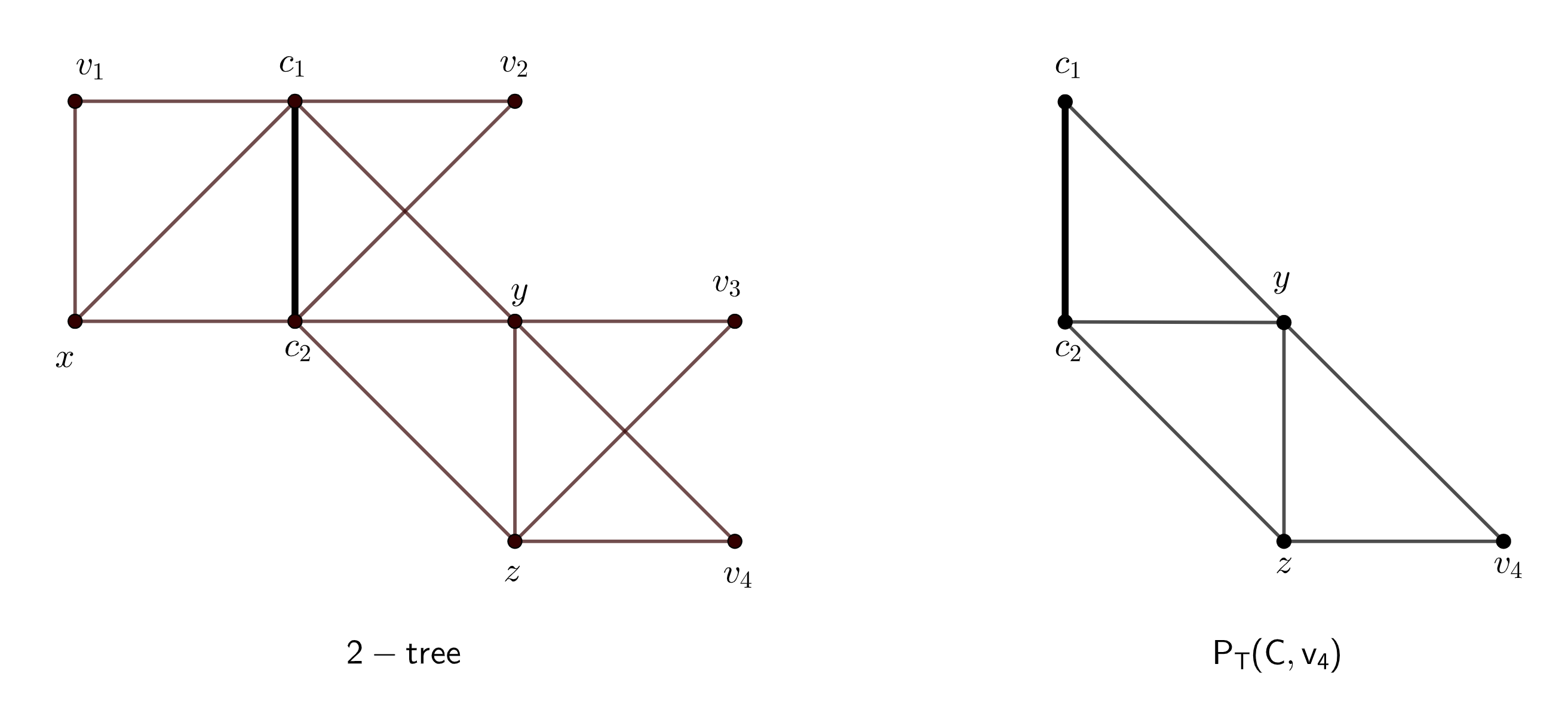}
	\caption{A $2$-tree and its path-type $2$-tree $P_T(C,v_4)$.}
	\label{fig3}
\end{figure}

Recall that $L(T)$ is the set of all $k$-leaves of $T$.
The following result,
	due to Stephens and Oellermann~\cite{Stephens},
shows that each $k$-tree can be represented as the union of path-type $k$-trees starting at a $k$-clique $C$ and ending at $k$-leaves.

\begin{lemma}[\cite{Stephens}]
	\label{decom}
        For any nontrivial $k$-tree $T$
        and any $k$-clique $C$, we have
\begin{align*}
        T=\left(\bigcup_{v\in L(T)
        	\setminus V(C)}
        V(P_T(C,v)),\bigcup_{v\in L(T)\setminus V(C)}
        E(P_T(C,v))\right).
\end{align*}
\end{lemma}
In the following, the characteristic 1-tree is defined.
For any $k$-clique $C$ of $T$
and any vertex $v\in V(T)\setminus V(C)$,
if $A_T(C,v)=(C,w_1,w_2,\cdots,w_s,v)$,
then  let $P'_T(C,v)$ denote
the path with vertex set
\begin{align*}
        V(P'_T(C,v))=\{C,w_1,w_2,\cdots,w_s,v\}
\end{align*}
and edge set
\begin{align*}
        E(P'_T(C,v))=\{Cw_1,w_1w_2,w_2w_3,\cdots,w_{s-1}w_s,w_sv\}.
\end{align*}

\begin{definition}
Let $T$ be a nontrivial $k$-tree and  $C$ be any $k$-clique in $T$.
The {\it characteristic $1$-tree} of $T$ with respect to $C$, is the graph $T'_C$ defined by
\begin{align*}
        \left(\bigcup_{v\in L(T)\setminus V(C)}
        V(P'_T(C,v)),
        \bigcup_{v\in L(T)\setminus V(C)}
        E(P'_T(C,v))\right).
\end{align*}
\end{definition}
It is obvious that the graph $T'_C$ is connected. Moreover the uniqueness of $A_T(C,v)$, by Lemma \ref{uniq}, guarantees that the graph $T'_C$ is acyclic. Thus the characteristic $1$-tree is a tree and is the unique tree such that $A_{T'_C}(C,v)=A_T(C,v)$ for all
$v\in L(T)\setminus V(C)$. We can take the convention that the characteristic $1$-tree of the trivial $k$-tree is $K_1$.

It is clear that $\lvert V(T) \rvert=\lvert V(T'_C) \rvert+k-1$.
Stephens and Oellermann~\cite{Stephens} found
a bijection between the sub-$k$-trees containing the clique $C$ of a $k$-tree $T$ and the subtrees of $T'_C$ containing the vertex $C$.
They also established a relation
between $\mu(T;C)$ and $\mu(T'_C;C)$
as stated below.

\begin{lemma}[\cite{Stephens}]
	\label{bilemma}
        For any $k$-tree $T$ containing a $k$-clique $C$, \begin{align}\notag
                \mu(T;C)=\mu(T'_C;C)+k-1.
        \end{align}
\end{lemma}

Two $k$-cliques in a $k$-tree $T$ are called \emph{adjacent} if they are contained in a $(k+1)$-clique of $T$.
In general, two characteristic $1$-trees of $T$ with respect to two different $k$-cliques $C_1$ and $C_2$ are unrelated.
However,
if $C_1$ and $C_2$ are adjacent,
we find that $T'_{C_2}$ can be obtained
from $T'_{C_1}$ by a partial Kelmans operation.

We first give some notations. For a $k$-clique $C$ in $T$,
let $N_T(C)$ denote the set of vertices
$v$ in $V(T)\setminus V(C)$
which is adjacent to all vertices of $C$.
For a $(k+1)$-clique $Q$ in $T$, we denote by $U_{T}(Q)$ the set of the vertices $v$
in $V(T)\setminus V(Q)$
such that $v\in N_T(C)$ for some
$k$-clique $C$ of $Q$.

\begin{lemma}\label{adlemma}
Let $T$ be a nontrivial $k$-tree,
and let $C_1$ and $C_2$ be two
adjacent $k$-cliques in $T$.
Then $T'_{C_2}$ is isomorphic to the tree obtained by some partial Kelmans operation on $T'_{C_1}$ from
$c_2$ to $C_1$,
where $c_{2}\in V(C_2)\setminus V(C_{1})$.
\end{lemma}

\begin{proof}
Let $Q$ denote the subgraph
induced by $V(C_1)\cup V(C_2)$.
As $C_1$ and $C_2$ are adjacent in $T$, $Q$ is a $(k+1)$-clique of $T$.
Let $c_1$ be the only vertex in
$V(C_1)\setminus V(C_2)$.
Then,
$V(Q)=V(C_i)\cup \{c_{3-i}\}$
for both $i=1,2$.
\begin{figure}[h]
	\centering
	\includegraphics[width=0.6\linewidth]{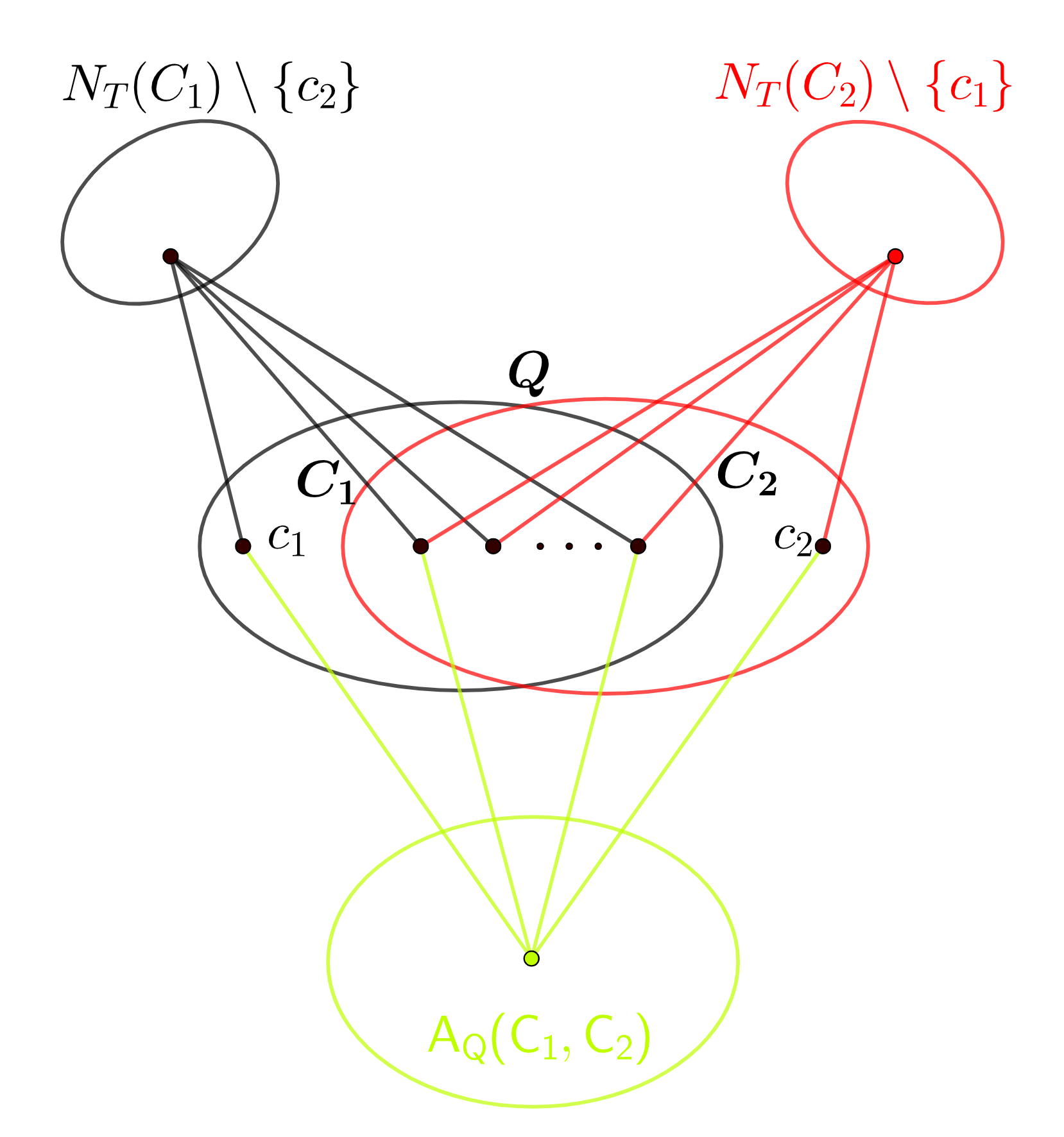}
	\caption{The neighbours of the $(k+1)$-clique $Q=C_1\cup C_2$.}
	\label{fig4}
\end{figure}
Let $A_Q(C_1,C_2)=U_{T}(Q)\setminus (N_T(C_{1})\cup N_T(C_{2}))$ (see Figure \ref{fig4}),
and let $\tilde{T}'_{C_1}$ be the tree obtained by the partial Kelmans operation on $T'_{C_1}$ from $c_2$ to $C_1$, which replaces the edge $c_2u$ by a new edge $C_1u$ for each $u\in A_Q(C_1,C_2)$,
as shown in Figure \ref{fig5}.

We will complete the proof by showing
that the mapping $f$ defined below is
an isomorphism  from
$\tilde{T}'_{C_1}$ to
$T'_{C_2}$:
$$
f(C_1)=c_1, \quad
f(c_2)=C_2,\quad
f(u)=u, \forall
u\in V(\tilde{T}'_{C_1})\setminus \{C_1,c_2\}.
$$

We consider the relationship between $T'_{C_1}$ and $T'_{C_2}$. For each vertex $u\in N_T(C_1)\setminus \{c_2\}$, $u$ is adjacent to $C_1$ in $T'_{C_1}$, and $u$ is also adjacent to $c_1$ in $T'_{C_2}$.  Similarly, for each vertex $v\in N_T(C_2)\setminus \{c_1\}$, $v$ is adjacent to $C_2$ in $T'_{C_2}$, and $v$ is also adjacent to $c_2$ in $T'_{C_1}$. For each vertex $w\in A_Q(C_1,C_2)$, $w$ is adjacent to $c_2$ in $T'_{C_1}$ and adjacent to $c_1$ in $T'_{C_2}$, respectively. See Figure \ref{fig5}.
\begin{figure}[h]
	\centering
	\includegraphics[width=0.9\linewidth]{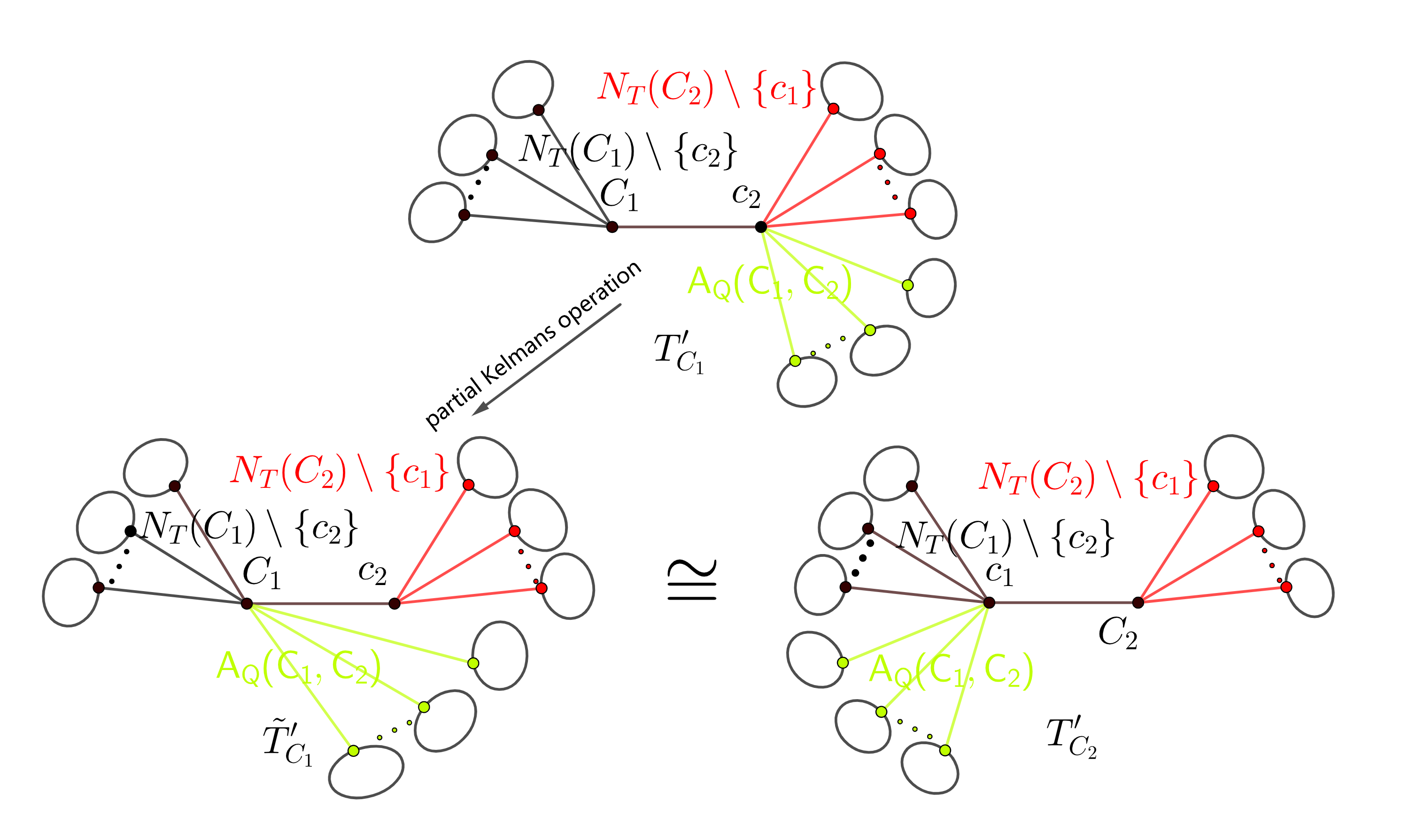}
	\caption{An illustration of the relationship between the two characteristic $1$-trees $T_{C_1}'$ and $T_{C_2}'$.}
	\label{fig5}
\end{figure}

For each vertex $v \in V(T)\setminus (U_T(Q)\cup V(Q))$,  by Lemma \ref{uniq}, there exists a unique path $P'_T(C_1,v)$ = $(C_1,v_1,v_2,\cdots,v_i,v)$ in the characteristic $1$-tree $T'_{C_1}$.

{\bf Case 1.} $v_1\in Q$.

In this case, $v_1=c_2$.

{\bf Subcase 1.1.} $v_2\in A_Q(C_1,C_2)$.

 In this case, $P'_{T}(C_2,v)$ = $(C_2,c_{1},v_2,v_3,\cdots,v_i,v)$ is also a unique path in the characteristic $1$-tree $T'_{C_2}$,
 corresponding to the unique path
 $(c_2,C_{1},v_2,v_3,\cdots,v_i,v)$
 in $\tilde{T}_{C_1}'$.

 {\bf Subcase 1.2.} $v_2\in N_T(C_2)\setminus \{c_1\}$.

 In this case, $P'_{T}(C_2,v)$ = $(C_2,v_2,v_3,\cdots,v_i,v)$ is a unique path in the characteristic $1$-tree $T'_{C_2}$,
  corresponding to the unique path
 	$(c_2,v_2,v_3,\cdots,v_i,v)$
 	in $\tilde{T}_{C_1}'$.

{\bf Case 2.} $v_1\notin Q$.

Since $v_1\notin Q$, we have $v_1\in N_T(C_1)\setminus \{c_2\}$. Then $P'_{T}(C_2,v)$ = $(C_2,c_1,v_1,v_2,\cdots,v_i,v)$ is a unique path in the characteristic $1$-tree $T'_{C_2}$,
corresponding to the unique path
	$(c_2,C_1, v_1,v_2,\cdots,v_i,v)$
	in $\tilde{T}_{C_1}'$.

Hence $f$ is an isomorphism
from $\tilde{T}_{C_1}'$ to
${T}_{C_2}$,
and the result holds.
\end{proof}

Figure \ref{fig2} shows an example of a $3$-tree and its characteristic $1$-trees with respect to adjacent $3$-cliques $C_1$ and $C_2$, respectively.
Let's introduce the following lemma that we need.
\begin{figure}[h]
\centering
\includegraphics[width=0.8\linewidth]{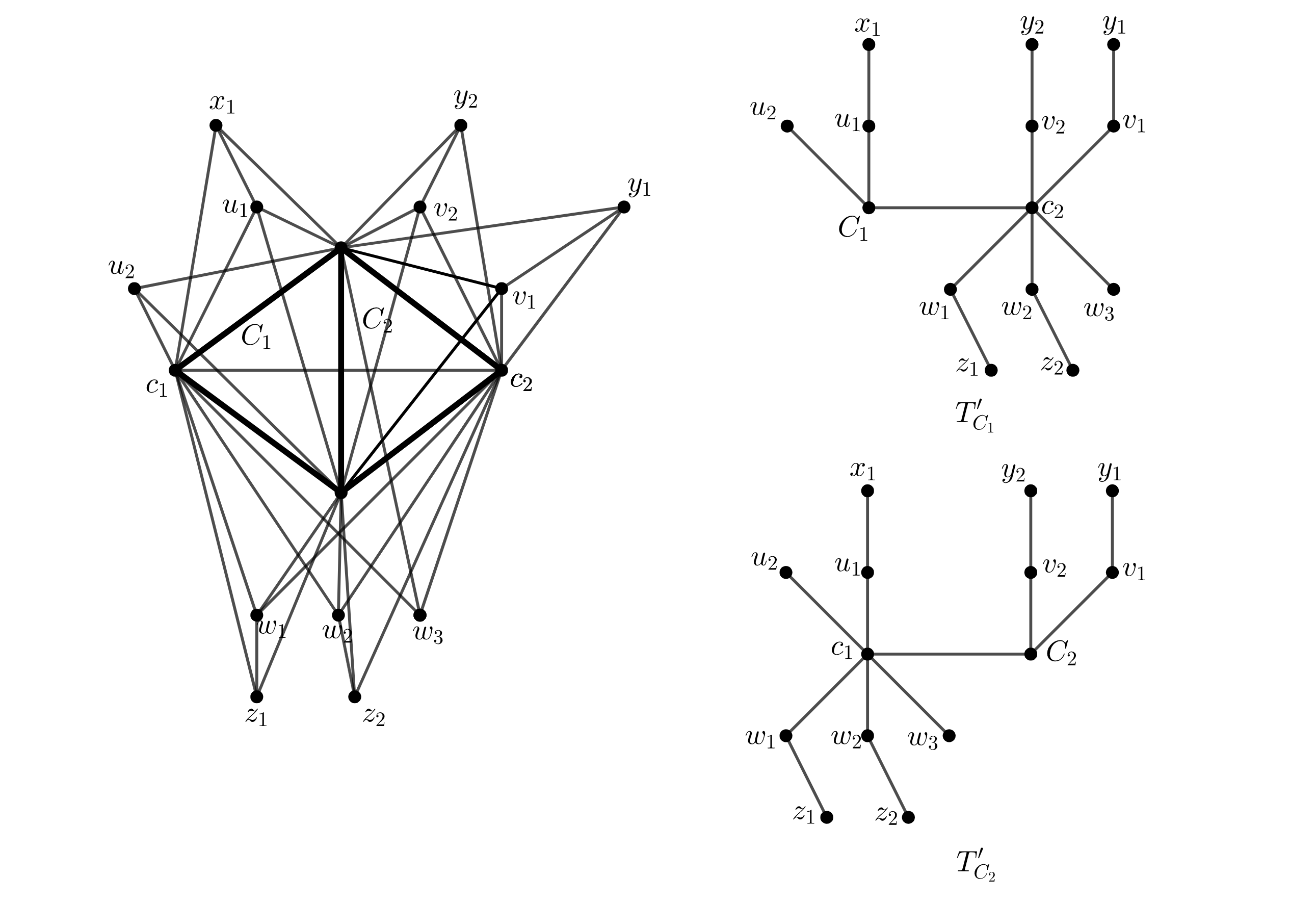}
\caption{A $3$-tree and its characteristic $1$-trees with respect to $3$-clique $C_1$ and $C_2$, respectively.}
\label{fig2}
\end{figure}

\begin{lemma}[\cite{Wagner2}]
	\label{treelemma}
Let $u$ be a vertex of degree at least 3 in a tree $T$. Then there exists a neighbour $v$ of $u$ such that
$\mu(T;v)>\mu(T;u)$.

\end{lemma}

We generalize Lemma \ref{treelemma} to $k$-trees.

\begin{theorem}\label{k-trees}
For any major $k$-clique $C_1$ of a $k$-tree $T$, there exists a $k$-clique $C_2$ adjacent to $C_1$ such that $\mu(T;C_2)> \mu(T;C_1)$.
\end{theorem}

\begin{proof}
We consider characteristic 1-tree $T'_{C_1}$. Since $deg_{T'_{C_1}}(C_1)\geq 3$, there exists a neighbour $c_{2}$ of $C_1$ in $T'_{C_1}$ such that $\mu(T'_{C_1};c_2)> \mu(T'_{C_1};C_1)$ by Lemma \ref{treelemma}.
Take $c_{1}\in C_1$, let $C_2=(C_1\setminus\{c_1\})\cup \{c_2\}$, then $C_2$ is a $k$-clique adjacent to $C_1$.  By Lemma \ref{adlemma}, $T'_{C_2}$ is isomorphic to the tree obtained by the partial Kelmans operation on $T'_{C_1}$ from $c_2$ to $C_1$. Thus, by Theorem \ref{patialtheorem}, we have
\begin{align*}
        \mu(T'_{C_2};C_2)\geq\mu(T'_{C_1};c_{2})>\mu(T'_{C_1};C_1).
\end{align*}
Then by Lemma \ref{bilemma}, we have
\begin{align*}
        \mu(T;C_2)>\mu(T;C_1).
\end{align*}
\end{proof}

Theorem \ref{MainTheorem} follows from Theorem \ref{k-trees} immediately.

\section{Concluding remarks}

By applying Theorem~\ref{mainKO}, we can easily derive the following conclusion
due to Wagner and Wang~\cite{Wagner2}.

\begin{proposition}[\cite{Wagner2}, \textbf{Lemma 3.1}]
\label{Col2.3}
Let $v$ be a leaf of a tree $T$, and let $u$ be a neighbour of $v$. Then
\[\mu(T;v)\geq \mu(T;u),\]
with equality if and only if $T$ is a path.
\end{proposition}

\begin{proof}
As $v$ is a leaf of $T$,  $T= T\Kel{v}{u}$.
By (\ref{eq2}), $\mu(T,v)=\mu(T\Kel{v}{u},v)\ge \mu(T,u)$.
\end{proof}

An \emph{end} $k$-clique is a $k$-clique in a $k$-tree $T$ with degree $1$, i.e., it is contained in a unique $(k+1)$-clique of $T$.  We now generalize Proposition \ref{Col2.3} from trees to $k$-trees.

\begin{proposition}\label{PRO}
	Let $C_1$ be an end $k$-clique of a $k$-tree $T$ and $C_2$ be a $k$-clique adjacent to $C_1$. Then
	\begin{align*}
			\mu(T;C_1)\geq \mu(T;C_2),	
	\end{align*}
with equality if and only if $C_2$ is an end $k$-clique or $T$ is a path-type $k$-tree and $C_1$ contains a $k$-leaf.
\end{proposition}
\begin{proof}
	Set $c_1$=$V(C_1)\setminus V(C_2)$ and $c_2$=$V(C_2)\setminus V(C_1)$. Let $\tilde{T}'_{C_2}$ be the tree obtained by the Kelmans operation on $T'_{C_2}$ from $c_1$ to $C_2$. Note that $T'_{C_1}\cong \tilde{T}'_{C_2}$ by an isomorphism mapping $C_1$ to $c_1$ and $c_2$ to $C_2$. By (\ref{eq2}) in Theorem \ref{mainKO}, $\mu(T'_{C_1};C_1)=\mu(\tilde{T}'_{C_2};c_1)\geq \mu(T'_{C_2};C_2)$, with equality if and only if $C_2$ is a leaf of $T'_{C_2}$ or $T'_{C_2}$ is a path with $c_1$ as its leaf. Then, by Lemma \ref{bilemma}, we have $\mu(T;C_1)\geq \mu(T;C_2)$, with equality if and only if $C_2$ is a leaf of $T'_{C_2}$ or $T'_{C_2}$ is a path with $c_1$ as its leaf, which implies $C_2$ is an end $k$-clique or $T$ is a path-type $k$-tree with $C_1$ containing a $k$-leaf.
\end{proof}
Theorem \ref{MainTheorem} shows that the maximum local mean order of a $k$-tree must occur at either an end $k$-clique or a $k$-clique of degree 2. The following example shows that it may occur at an end $k$-clique.

We call the initial $k$-clique in the recursive definition of a $k$-tree the \emph{base} $k$-clique.
\begin{definition}	(star-type $k$-tree) Fix an integer $k\geq1$.	
\begin{description}		
\item[(1)] The complete graph $K_k$ is a star-type $k$-tree.		
\item[(2)] If $S$ is a star-type $k$-tree, then so is the graph obtained from $S$ by joining a new vertex to the base $k$-clique.	
\end{description}
\end{definition}
Let $S_{n+k}$ be a star-type $k$-tree with order $n+k$ with $n\geq 3$. Suppose that $V(S_{n+k})=\{a_1,a_2,\cdots,a_k,b_1,b_2,\cdots,b_n\}$, where $C=\{a_1,a_2,\cdots,a_k\}$ is the base clique. Let $T_n$ be the $k$-tree obtained from $S_{n+k}$ by joining a new vertex $c_i$ adjacent to each vertex of the set $\{b_i\}\cup (C\setminus \{a_i\})$ for each $i=1,2,\cdots,n$. An example $T_3$ is shown in Figure \ref{fig6}. To the $k$-tree $T_n$, each $c_i$ is a $k$-leaf, the degree of the base clique $C$ is $n$ and each $k$ clique $\{b_i\}\cup (C\setminus \{a_i\})$ has degree 2, other $k$-cliques are all end $k$-cliques. By Theorem \ref{k-trees} and Proposition \ref{PRO}, the maximum local mean order of $T_n$ must occur at a $k$-clique of degree $1$.
\begin{figure}[!t]
\centering
\includegraphics[width=0.8\linewidth]{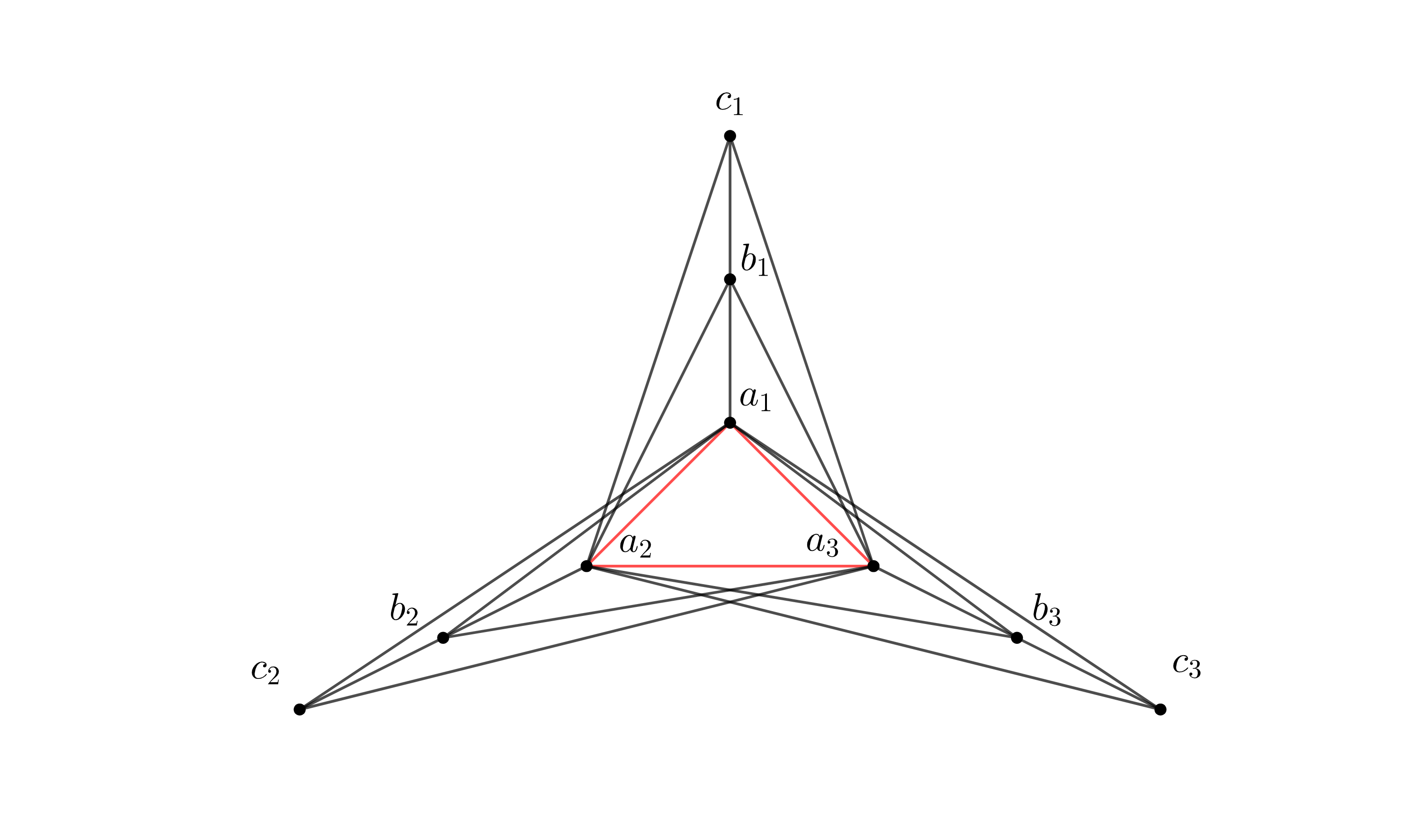}
\caption{A 3-tree $T_3$.}
\label{fig6}
\end{figure}

We now consider $k$-trees whose maximum local mean order
occurs at a k-clique of degree $2$. For trees, it has been pointed out in \cite{Wagner2}
that there are infinitely many trees for which the maximum local mean order
occurs at an vertex of degree $2$.
For example, if $T$ is the tree 
obtained from $P_{2n+1}$ by attaching two pendant edges to each end of $P_{2n+1}$, 
as shown in Figure~\ref{fig9}, 
the maximum local mean of $T$ 
occurs at a vertex of degree $2$ 
whenever $n\ge 7$ (see \cite{Wagner2}). 
For $k$-trees with $k\geq 2$, we have not found such  examples and it deserves further studying.

\begin{figure}[h]
	\centering
	\includegraphics[width=0.7\linewidth]{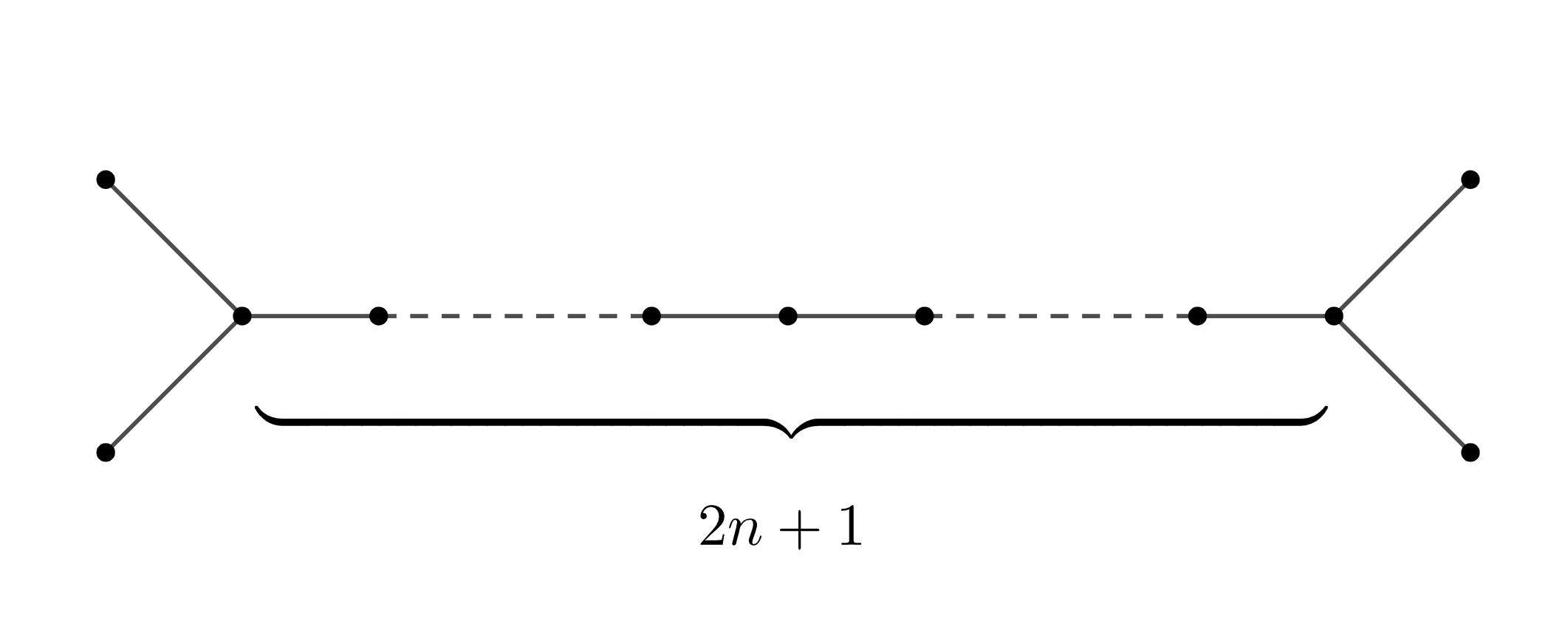}
	\caption{A tree obtained by attaching two pendant edges at each end of $P_{2n+1}$.}
	\label{fig9}
\end{figure}

\begin{problem}
	Is there a $k$-tree $T$ for $k\geq2$ such that the maximum local mean order of sub-$k$-trees containing a given $k$-clique occurs at a $k$-clique of degree 2, but not at an end $k$-clique?
\end{problem}

\section*{Acknowledgements}
This work is supported by NSFC (Nos. 12171402 and 12371340) and
the Ministry of Education,
Singapore, under its Academic Research Tier 1 (RG19/22).
Any opinions,
findings and conclusions or recommendations expressed in this
material are those of the authors and do not reflect the views of the
Ministry of Education, Singapore.

\end{document}